\documentclass{article}
\usepackage{amsmath}
\usepackage{amsfonts}
\usepackage{hyperref}
\usepackage{amssymb}
\usepackage{graphicx}
\providecommand{\U}[1]{\protect\rule{.1in}{.1in}}
\newtheorem{theorem}{Theorem}

\newtheorem{corollary}[theorem]{Corollary}

\newtheorem{definition}[theorem]{Definition}

\newtheorem{lemma}[theorem]{Lemma}

\newtheorem{proposition}[theorem]{Proposition}
\newtheorem{remark}[theorem]{Remark}

\newenvironment{proof}[1][Proof]{\noindent\textbf{#1.} }{\ \rule{0.5em}{0.5em}}
\begin{document}

\title{$Local^{3}$ Index Theorem.}
\author{\large Nicolae Teleman\\
Dipartimento di Scienze Matematiche, \\
Universita' Politecnica delle Marche, 60131-Ancona, Italia\\
 e-mail:   teleman@dipmat.univpm.it}

\maketitle

\tableofcontents

%%%%%%%%%%%%%%%%%%%%%%%%%%%%%%%%%%%%%%%%%%%%%%%%%%%%%%%%%%%%%%%%%%%

\section{Abstract}   %%%%%%% Sect.1
\par
$Local^{3}$ Index Theorem means $Local(Local(Local \;Index \; Theorem)))$.  
\par
\noindent$Local \; Index \; Theorem$ is the
Connes-Moscovici local index theorem \cite{Connes-Moscovici1}, \cite{Connes-Moscovici2}.  The second ''Local'' refers to the  cyclic homology localised to a certain 
separable subring of the ground algebra, while the last one refers to Alexander-Spanier type cyclic homology.
\par  
The Connes-Moscovici work is based on the operator $R(A) = \mathbf{P} - \mathbf{e}$ associated to the elliptic pseudo-differential operator $A$ on the smooth manifold $M$, where $\mathbf{P}$ ,  $\mathbf{e}$ are idempotents, see \cite{Connes-Moscovici1}, Pg. 353.
\par
The operator $R(A)$ has two main merits: it is a smoothing operator and its distributional kernel is situated in an arbitrarily small neighbourhood of the diagonal in $M \times M$.  \par
The operator $R(A)$ has also two setbacks: -i) it is not an idempotent (and therefore it does not have a genuine Connes-Chern character); -ii) even if it were an idempotent, its Connes-Chern character would belong to the cyclic homology of the algebra of smoothing operators (with \emph{arbitrary} supports, which is \emph{trivial}. 
\par
This paper presents a new solution to the difficulties raised by the two setbacks. 
\par
For which concerns -i), we show that although $R(A)$ is not an idempotent, it satisfies the identity
 $
 (\mathbf{R}(A))^{2} \;=\; \mathbf{R}(A) - [\mathbf{R}(A) . e + e .\mathbf{R}(A) ].
 $
 We show that the operator $R(A)$ has a genuine Chern character provided the cyclic homology complex of the algebra of smoothing operators is \emph{localised} to the separable sub-algebra $\Lambda = \mathbb{C} + \mathbb{C} . e$, see Sect. 8.1.
\par
For which concerns -ii), we introduce the notion of \emph{local} cyclic homology; this is constructed on the foot-steps of the Alexander-Spanier homology, i.e. by filtering the chains of the cyclic homology complex of the algebra of smoothing operators by their distributional support, see Sect. 7.
\par
Using these new instruments, we give a reformulation of the Connes-Moscovici local Index Theorem, see Theorem 23, Sect. 9. As a corollary of this theorem, we show that the \emph{local} cyclic homology of the algebra of smoothing operators is at least as big as the Alexander-Spanier homology of the base manifold.
\par
The present reformulation of Connes-Moscovici local index theorem  opens the way to new investigations, see Sect. 10.    

 \section{Indroduction}       %%%%%%%%%  Sect. 2
Using the language of non-commutative geometry, Connes and Moscovici \cite{Connes-Moscovici1}, \cite{Connes-Moscovici2}  build an algebraic bridge connecting in a natural way the analytical index and the topological
index of elliptic pseudo-differential operators on smooth manifolds. Their construction extends also to topological manifolds with quasi-conformal structure, see Connes-Sullivan-Teleman \cite{Connes-Sullivan-Teleman}.  
 \par
 Given an elliptic pseudo-differential operator $A$ on the smooth manifold $M$, Connes and Moscovici associate
 its index class $Ind (A) \in H^{AS}_{\ast}(M)$, where $H^{AS}_{\ast}(M)$ denotes Alexander-Spanier homology.
 The index class is obtained as the result of the composition of two constructions
\begin{equation}
K^{1}(C(S^{\ast}M)) \stackrel{\mathbf{R}}{\longrightarrow}  K^{0} ({\mathcal{K}}_{M}) \stackrel{\tau}{\longrightarrow}  H_{ev}^{AS} (M),
\end{equation}
where $K^{1}(C(S^{\ast}M)) $  contains the symbol of the elliptic operator $A$;
 $K^{0} ({\mathcal{K}}_{M}) $ consists of differences of stably homotopy classes of smooth idempotents with \emph{arbitrary} supports.
 \par
 The first homomorphism $R$ applied upon the operator $A$ is given by an algebraic construction involving the operator $A$ and one of its parametrices $B$; as a result one obtains an operator $\mathbf{R}(A)$ who has the following basic properties:
 \par
 -1) it has small support about the diagonal in $M \times M$
 \par
 -2) it is a smoothing operator on $M$
 \par
 -3) $\mathbf{R}(A) = \mathbf{P} - \mathbf{e}$, where $\mathbf{P}$ and  $\mathbf{e}$ are idempotents, with small support about the diagonal. The idempotent $\mathbf{e}$ is a constant operator; it does not 
 contain homological information.
 \par
The smoothing operator $\mathbf{R}(A)$ is obtained by implementing the connecting homomorphism 
$\delta_{1}: K^{1}(C(S^{\ast}M)) \longrightarrow  K^{0} ({\mathcal{K}}_{M})$ associated to the short exact sequence of Banach algebras
\[
0 \rightarrow {\mathcal{K}}_{M} \rightarrow   {\mathcal{L}}_{M}   \rightarrow  C(S^{\ast}M) 
\rightarrow  0.
\]
The operator $\mathbf{R}(A)$ has two inconveniences:
\par
-i) it is not an idempotent
\par
-ii) even if $\mathbf{R}(A)$ were an idempotent, its Chern character, belonging to the cyclic
homology of the algebra of smoothing operators with \emph{arbitrary} supports, would be trivial because this cyclic homology is trivial.
\par
Connes-Moscovici \cite{Connes-Moscovici1}, Pg. 352 state clearly  
that the connecting homomorphism $\partial_{1}$ takes values in $ K^{0} ({\mathcal{K}}_{M}) \simeq \mathbb{Z}$ and that the information carried by it is solely the \emph{index} of the operator. However, \cite{Connes-Moscovici1} states also that by pairing the residue operator $\mathbf{R}(A)$ with the Alexander-Spanier cohomology (which is \emph{local}) one recovers the whole co-homological information carried by its symbol, (see the Connes-Moscovici local Index Theorem 1 of Sect. 4). 
\par
The Connes-Moscovici index class $Ind (A) := \tau (\; \mathbf{R}(A)  \;)$ is a  well defined Alexander-Spanier \emph{homology class} on $M$. Its correctness depends upon two important ingredients:
\par
-a) the realisation of the Alexander-Spanier co-homology by means of skew-symmetric co-chains;
this allows one to get rid of the idempotent $\mathbf{e}$ from the expression of  $\mathbf{R}(A)$ and ultimately to treat  $\mathbf{R}(A)$ as it were an idempotent, see -3); call it virtual idempotent.
\par
-b) at this point, $Ind (A)$ uses the formal pairing of the Chern character of the virtual idempotent -a) 
with the cyclic homology of the algebra $C^{\infty}(M)$.  This operation requires to produce trace class
operators.
 \par
 In this paper we address  the same problem, i.e. to define an algebraic bridge between the analytical and topological index of elliptic operators. However, we propose here a different way to overcome
 the difficulties described above. This will be done by introducing two main ideas. 
 \par
 The first idea of the paper is based on the remark that the residue operator $\mathbf{R}(A)$  satisfies the identity
 \[
 (\mathbf{R}(A))^{2} \;=\; \mathbf{R}(A) - [\mathbf{R}(A) . e + e .\mathbf{R}(A) ].
 \]
 We show in Sect. 8.1. that, based on this identity, the operator $\mathbf{R}(A)$  has a genuine Chern character provided the ordinary cyclic homology complex of the algebra of smoothing operators
 is localised with respect to the  separable ring $\Lambda := \mathbb{C} + \mathbb{C} e$. This replacement does not modify the cyclic homology.
 \par
 The second idea of the paper consists of replacing the ordinary cyclic homology of the algebra of smoothing operators by the \emph{local} cyclic homology of the algebra.  This is done by filtering the cyclic complex of the algebra based on the supports of the chains. The \emph{local} cyclic homology
 of the algebra of smoothing operators is then defined in the same way as the Alexander-Spanier co-homology is defined.
 \par
The combination of these two ideas allows one to reformulate the Connes-Moscovici local index theorem, see Theorem 23.
\par
Notice that our considerations do not require necessarily to deal with trace class operators; this opens the way to new applications.
\par
As a corollary of Theorem 23 we obtain that the \emph{local} cyclic homology
of the algebra of smoothing operators is at least as big as the Alexander-Spanier homology of the base space, see Proposition 25.
\par
We stress also that our \emph{local} cyclic homology of the Banach algebra of smoothing operators is independent of  Connes' notions of \emph{entire} or \emph{asymptotic} cyclic homology, see \cite{Connes1}, \cite{Connes2} and differs
from  Puschnigg's \cite{Puschnigg} construction. 
\par
Our methods lead to interesting questions and new scenarios; these will be addressed elsewhere, see Sect. 10.
\par
The author thanks Jean-Paul Brasselet, Andr$\acute{e}$ Legrand and Alexandr Mischenko for useful conversations. 
 \par  
%%%%%%%%%%%%%%%%%%%%%%%%%%%%%%%%%%%%%%%%%%%%%%%%%%%%%%%%%%%
\section{Recall of $K$-theory groups of Banach algebras.}   %%%%%%%%  Sect. 3 (OK)
\par
For the benefit of the reader we recall here the basic definitions regarding the $K$-theory of Banach algebras.
\par
To begin with, as motivation, suppose $X$ is a compact connected topological space and $C(X)$ denote the $C^{\ast}$ algebra of continuous complex valued functions on $X$. A continuous complex vector bundle over $X$ may be described either as a finite projective module over $C(X)$, or as an idempotent of the matrix algebra  $\mathcal{M}_{n}(C(X))$. Passing to isomorphism classes of bundles allows one to identify bundles over $X$ with continuous homotopy classes of idempotents of the algebra $\mathcal{M}_{n}(C(X))$, with $n$ sufficiently large. Denote by $\textit{Vect}(X)$ the set of such homotopic classes of idempotents.
\par
The direct sum of finite projective $C(X)$ modules passes to  $\textit{Vect}(X)$ so that it becomes a commutative semigroup. Taking the Grothendieck completion of this semigroup, one obtains the $K-$theory group $K^{0}(X)$. Any element of the group $K^{0}(X)$ may be represented as $ \xi = [p] - [q]$, where $[p]$ and $[q]$  are the homotopy classes of two idempotents $p, q \in \mathcal{M}_{m}(C(X))$.
The idempotent $q$ may be chosen to be a unit matrix. The element $\xi$ is called a virtual vector bundle of rank $m-n$.
\par
The subgroup of $K^{0}(X)$ consisting of virtual bundles of rank zero is denoted by $\tilde{K}^{0}(X).$
\par
An idempotent $p_{n} \in \mathcal{M}_{n}(C(X)) $ may be used to produce the idempotent $p_{n+1} \in \mathcal{M}_{n+1}(C(X)) $ by  stabilization
\begin{equation}
p_{n+1} :=
\left( \begin{array}{cc}
p_{n}&0\\
0&1
\end{array}
\right)
\end{equation}
\par
Then any element of $\tilde{K}^{0}(X)$ may be thought of as the stably homotopy class of an idempotent in $\mathcal{M}_{m}(C(X))$.
\par
As for any general homology functor, one defines $K^{1}(X) = \tilde{K}^{0} (\Sigma X)$, where $\Sigma$ denotes suspension.
Any vector bundle over $\Sigma X$ may be described by a clutching function $f: X \rightarrow GL(m, \textsl{C})$.
Alternatively, any such function may be thought of as an element of the subset $GL_{m}(C(X)) \subset \mathcal{M}_{m}(C(X))$
consisting precisely of all invertible elements of the algebra $M_{m}(C(X))$.
\par
As the stabilization formula above may be used not only for idempotents but also for invertibles, one gets an equivalent definition of $K^{1}(X)$
\begin{equation}
K^{1}(X) = stably \hspace{1mm} homotopy \hspace{1mm} classes \hspace{1mm} of \hspace{1mm} invertibles \hspace{1mm} in \hspace{1mm}
\mathcal{M}_{m}(C(X)),
\end{equation}
or
\begin{equation}
K^{1}(X) = \pi_{0}\hspace{1mm}( \hspace{1mm}Lim_{m \rightarrow \infty} GL_{m}((C(X)) \hspace{1mm}).
\end{equation}
\par
The higher order $K$ theory groups are defined by
\begin{equation}
K^{i} (X) = \tilde{K} ({\Sigma}^{i} X),  \hspace{3mm} 1 \leq i.
\end{equation}
\par
For any closed subspace $ Y \subset X$, one define the relative $K$ groups
\begin{equation}
K^{i} ( X, Y) := \tilde{K}^{i} (X/Y),
\end{equation}
where $X/Y$ denotes the quotient space.
\par
The Bott periodicity theorem implies the periodicity of the $K$-theory groups, which leads to the
6-term exact sequence

%%%%%%%%%%%%%%%%%%%%%%%%%%%%%%%%%%%%%%%%%%%%%%%%%%%%%%%%%%%%%%%%%%%%%  6-Term exact sequence -begin
\begin{equation}
 \begin{array}{ccccc}
 K^{0}(X/Y)& \rightarrow & K^{0}(X)& \rightarrow  &  K^{0}(Y) \\
 \uparrow  &             &         &              & \downarrow \\
 K^{1}(Y)  & \leftarrow  & K^{1}(X)&  \leftarrow  &  K^{1}(X/Y)
\end{array}.
\end{equation}

%%%%%%%%%%%%%%%%%%%%%%%%%%%%%%%%%%%%%%%%%%%%%%%%%%%%%%%%%%%%%%%%%%%%%  6-Term exact sequence -end
\par
If we replace in the above constructions the algebra $C(X)$ by an arbitrary unital Banach algebra $\cal{A}$,
the $K$ theory groups of the algebra $\cal{A}$ are defined (with $[\;]$ meaning homotopy class)
\begin{equation}
K^{0}({\cal A}) := \{\;[p] - [q] \;| \; p^{2}=p \in  \mathcal{M}_{m}({\cal A}), \hspace{2mm} q^{2}=q \in
\mathcal{M}_{n}({\cal A}), \;\; m, n \in \textit{N} \;\}
\end{equation}
\begin{equation}
K^{1}({\cal A}) := \pi_{0}\hspace{1mm}( \hspace{1mm}Lim_{m \rightarrow \infty} GL_{m}({\cal A}) \hspace{1mm}).
\end{equation}
\par
To the $(X,\;Y)$ pair of compact non-empty topological spaces there corresponds the exact sequence of algebras of continuous
functions
\begin{equation}
0  \rightarrow   C(X,\;Y) \rightarrow  C(X)   \rightarrow C(Y)   \rightarrow  0,
\end{equation}
where
\begin{equation}
C(X,\;Y) \;=\; \{f \;\in\; C(X) \; | \;  f(Y) = 0  \; \}.
\end{equation}
More generaly, one may consider an arbitrary exact sequence of Banach algebras
\begin{equation}
0  \rightarrow   J \rightarrow  A   \rightarrow B \rightarrow  0.    %%%%%% Formula (11)
\end{equation}
Let $\tilde{J} := J \oplus \textit{C}.1$ be the algebra $J$
with the unit $ 1 \in \textit{C}$ adjoined and let $\epsilon: \tilde{J} \rightarrow \textit{C}$ be the augmentation mapping.
By definition, for $i = 0, 1$
\begin{equation}
K^{i} (J) = Ker \;{\epsilon}_{i}
\end{equation}
where ${\epsilon}_{i}:  K^{i} (\tilde{J})  \rightarrow   K^{i} (\textit{C}).$
\par
By construction, $K^{i}(C(X)) = K^{i}(X),$   $i = 0, 1.$
\par
 The analogue of Bott periodicity holds for $K$-theory groups, see Wood \cite{Wood}; hence,
for any short exact sequence of Banach algebras and continuous mappings as above, the
6-terms $K$ exact sequence of $K$-theory groups holds
%%%%%%%%%%%%%%%%%%%%%%%%%%%%%%%%%%%%%%%%%%%%%%%%%%%%%%%%%%%%%%%%%%%%%  6 - Term exact sequence K-homology -begin

\begin{equation}
 \begin{array}{ccccc}
 K^{1}(J)& \rightarrow & K^{1}(A)& \rightarrow  &  K^{1}(B) \\
 \uparrow  &             &         &              & \downarrow \\
 K^{0}(B)  & \leftarrow  & K^{0}(A)&  \leftarrow  &  K^{0}(J)
\end{array}.
\end{equation}

%%%%%%%%%%%%%%     6 - Term exact sequence K-homology -end

%%%  %%%%  Section 4
\section{Connes-Moscovici Local Index Theorem.}    %%%%%%   Sect. 4 
\par
In this section we summarize the Connes-Moscovici \cite{Connes-Moscovici1} construction of the
{\em local index class} for an elliptic operator. All constructions and notations 
in this section are those of \cite{Connes-Moscovici1}.
\par
To fix the notation, let $A: L_{2}(E) \rightarrow L_{2}(F) $ be an elliptic pseudo-differential operator of order zero
from the vector bundle $E$ to the vector bundle $F$  on the compact smooth manifold $M$.
 Let ${\sigma}_{pr} (A) = a$ be its principal symbol, seen as a continuous isomorphism from the bundle $\pi ^{\ast} E$ to the bundle $\pi ^{\ast} F$ over the unit co-sphere bundle $S(T^{\ast}M)$  ( $\pi$ is the co-tangent bundle projection).
Let $B$ be a pseudo-differential parametrix for the operator $A$. The parametrix $B$, having principal symbol 
${\sigma}_{pr} (B) = a^{-1},$ may be chosen so that the operators $S_{0}= 1 - BA$  and $S_{1}= 1 - AB$ be {\emph smoothing} operators. Additionally, supposing that the distributional support of the operator $A$ is sufficiently small about the diagonal,
the operators $B$, $S_{0},$ $S_{1}$ may be supposed to have also small supports about the diagonal.
\par
With the operators $A, B, S_{0}, S_{1}$ one manufactures the invertible operator
\begin{equation}
\mathbf{L} =
\left( \begin{array}{cc}
S_{0}&-(1 + S_{0})B\\
A&S_{1}
\end{array}
\right)
: L_{2}(E) \oplus L_{2}(F) \rightarrow L_{2}(E) \oplus L_{2}(F)
\end{equation}
with inverse
\begin{equation}
\mathbf{L}^{-1} =
\left( \begin{array}{cc}
S_{0}&(1 + S_{0})B\\
-A&S_{1}
\end{array}
\right)
\end{equation}
The operator $\mathbf{L}$ is used to produce the idempotent $\mathbf{P}$
\begin{equation}
\mathbf{P} =
\mathbf{L} \left(
\begin{array}{cc}
1&0\\
0&0
\end{array}
\right) \mathbf{L}^{-1}.
\end{equation}
Let $\mathbf{P}_{1}$, resp.  $\mathbf{P}_{2}$ be the projection onto the direct summand
$L_{2}(E)$, resp.  $L_{2}(F)$.

A direct computation shows that 
\begin{equation}
\mathbf{R} := \mathbf{P} - \mathbf{P}_{2}=
\left( \begin{array}{cc}
S_{0}^{2}& S_{0}(1 + S_{0})B\\
S_{1}A&- S_{1}^{2}
\end{array}
\right).
\end{equation}
\noindent
This shows that the \emph{residue} operator $\mathbf{R}$ is a smoothing operator on
$L_{2}(E) \oplus L_{2}(F)$ with small support about the diagonal.
\par
The motivation for the consideration of the operators $L, P, R$ comes
from the implementation of the connecting homomorphism
\begin{equation}
\partial_{1}: K^{1}(C(S^{\ast}M)) \rightarrow K^{0}({\mathcal{K}}_{M})\cong  \texttt{Z}
\end{equation}
in the 6-terms K-theory  groups exact sequence associated to the short exact sequence
of $C^{\ast}$ algebras
\begin{equation}
0 \rightarrow {\mathcal{K}}_{M} \rightarrow   {\mathcal{L}}_{M}   \rightarrow  C(S^{\ast}M) 
\rightarrow  0;    %%%%%% formula (19)
\end{equation}
here ${\mathcal{K}}_{M}$ is the algebra of compact operators on $L_{2}(M),$  ${\mathcal{L}}_{M}$ is the
norm closure in the algebra of bounded operators of the algebra of pseudo-differential
operators of order zero and  $C(S^{\ast}M)$ is the algebra of continuous functions
on the unit co-sphere bundle to $M$.
\par
In fact, if $\mathbf{a}$ denotes the symbol of the elliptic operator $A$, then after embedding the
bundles $E$ and $F$ into the trivial bundle of rank $N$, the operators  $L, \mathbf{P}, \mathbf{P}_{2}$, resp. $R$,  may be seen as elements of the
matrix algebras $\mathcal{M}_{N}({\mathcal{L}}_{M})$, resp. $\mathcal{M}_{N}({\mathcal{K}}_{M})$, and hence
\begin{equation}
\partial_{1} ([\mathbf{a}]) = [\mathbf{P}] - [\mathbf{P}_{2}] \in K^{0} ({\mathcal{K}}_{M}).
\end{equation}
\par
Let $C^{q}(M)$ denote the space of Alexander-Spanier cochains of degree $q$ on $M$ consisting
of all smooth, \textit{anti-symmetric} real valued functions $\phi$ defined on $M^{q+1}$,
which have support on a sufficiently small tubular neighbourhood of the diagonal.
\par
Then, for any $[\mathbf{a}] \in K^{1}(C(S^{\ast}M))$,
and for any even number $q$ one considers the linear functional
\begin{equation}
\tau_{\mathbf{a}}^{q}: C^{q}(M) \longrightarrow   \mathbb{C}
\end{equation}
given by the formula
\begin{equation}
\tau_{\mathbf{a}}^{q} (\phi) = \int_{M^{q+1}} \mathbf{R} (x_{0}, x_{1}) \mathbf{R} (x_{1}, x_{2}) ...
\mathbf{R} (x_{q}, x_{0}) \phi (x_{0}, x_{1}, ... , x_{q}),
\end{equation}
where $\mathbf{R} (x_{0}, x_{1})$ is the kernel of the smoothing operator $\mathbf{R}$ defined above.
%%%%%%%%%%%%%%
\par
Using the above construction, Connes and Moscovici \cite{Connes-Moscovici1}
produce the \emph{index class homomorphism}
\begin{equation}
Ind: K^{1}(C(S^{\ast}M))  \otimes_{\mathbb{C}}   H^{ev}_{AS}(M) \longrightarrow \mathbb{C},
\end{equation}
where $H^{ev}_{AS}(M)$ denotes Alexander-Spanier cohomology. On the Alexander-Spanier co-chains $\phi$ 
it is defined by
\begin{equation}
Ind (\mathbf{a} \otimes_{\mathbb{C}}    \phi   ) :=  \tau_{\mathbf{a}}^{q} (\phi) 
\end{equation}

%%%%%%%%%%%%%%%
\par
The functional $\tau_{\mathbf{a}}^{q}$ is an Alexander-Spanier cycle of degree $q$
over $M$;  it defines a homology class $[\tau_{\mathbf{a}}^{q}] \in H_{q}(M, R)$.  

\begin{theorem} Connes-Moscovici \cite{Connes-Moscovici1} Theorem 3.9.
Let $A$ be an elliptic pseudo-differential operator on $M$ and let $ [\phi] \in H^{2q}_{comp}(M)$. Then
\begin{equation}
Ind_{[\phi]} A \;=\; \frac{1}{(2\pi i)^{q}} \frac{q!}{(2q)!}(-1)^{dim M} <\;Ch \sigma (A) \tau (M)  [\phi], \;[T^{\ast}M] \;>
\end{equation} 
where $\tau (M) = Todd (TM) \otimes C$ and $ H^{\ast} (T^{\ast}M)$ is seen as a module over $H^{\ast}_{comp}(M)$.
\end{theorem}

%%%%%%%%%%%%%%%%%%%%%%%%%   Sect. 5
\section{$K$-Theory \emph{Local} Symbol Index Class.}   %%%%%%    Sect. 5
\par
The content of this section presents an interest by itself although it is not going to be used in this paper.
\par
In this section we are going to show that by replacing in the above constructions the Hilbert spaces by corresponding bundles and operators by their symbols, one gets a quasi-local \emph{residue bundle} $\mathbf{r}$ on the total space of the co-tangent bundle $\pi: T^{\ast}(M) \longrightarrow M$. The Connes-Moscovici \emph{residue} operator $\mathbf{R}$ appears to be the quantification of the bundle $\mathbf{r}$. 
\par
All considerations here are made on the total space of the co-tangent bundle. 
\par
 Let $\xi = {\pi}^{\ast}(E)$ and $\eta = {\pi}^{\ast}(F)$ the pullbacks of the bundles $E$  and  $F$. Let $\lambda: T(M) \longrightarrow [0,1]$ be a smooth function which is identically zero on a small neighbourhood $U$ of the zero section and identically 1 on the complement of $2U$. 
\par
Let $A: L_{2}(E) \rightarrow L_{2}(F) $ be the elliptic pseudo-differential operator considered above
 and let $B$ be the pseudo-differential parametrix for the operator $A$. 
The principal symbols  of the operators $A$ and $B$ 
\begin{equation}{\sigma}_{pr} (A) = \mathbf{a}: \xi \longrightarrow \eta 
\end{equation} 
\begin{equation}
\mathbf{b} = {\sigma}_{pr} (B) = \mathbf{a}^{-1}: \eta \longrightarrow \xi
\end{equation}
are isomorphisms away from the zero section. The symbol $\mathbf{a}$ defines the $K^{0}$-theory triple 
$(\; \xi, \;\eta, \; \mathbf{a})$ with compact support on $T^{\ast}(M)$.
\par
We regularize the bundle homomorphisms $\mathbf{a}$, and $\mathbf{a}^{-1}$ by multiplying them by the function $\lambda$; let $\tilde{\mathbf{a}}$  and  $\tilde{\mathbf{b}}$ be the obtained bundle homomorphisms.
\par
 Let $s_{0}= 1 - \tilde{\mathbf{b}} \; \tilde{\mathbf{a}}$  and $s_{1}= 1 - \tilde{\mathbf{a}} \; \tilde{\mathbf{b}}$. They are bundle homomorphisms with supports in the neighbourhood $2U$ of the zero section.
\par
With the bundle homomorphisms $\tilde{\mathbf{a}}, \tilde{\mathbf{b}}, s_{0}, s_{1}$ one manufactures the smooth bundle isomorphism 
\begin{equation}
\textit{l} =
\left( \begin{array}{cc}
s_{0}&-(1 + s_{0})\tilde{\mathbf{b}}\\
\tilde{\mathbf{a}}&s_{1}
\end{array}
\right)
: \xi \oplus \eta \rightarrow \xi \oplus \eta
\end{equation}
with inverse

\begin{equation}
\textit{l}^{-1} =
\left( \begin{array}{cc}
s_{0}&(1 + s_{0})\tilde{\mathbf{b}}\\
-\tilde{\mathbf{a}}&s_{1}
\end{array}
\right)
\end{equation}

Let $\mathbf{p}_{1}$, resp.  $\mathbf{p}_{2}$,  be the direct sum projection of the
bundle  $\xi \oplus \eta$ onto the first, resp. the second, summand.

The isomorphism  $\textit{l}$ is used to produce the idempotent $\mathbf{p}$
\begin{equation}
\mathbf{p} =
\textit{l} \left(
\begin{array}{cc}
1&0\\
0&0
\end{array}
\right) \textit{l}^{-1} \; = \; \textit{l} \;  \mathbf{p}_{1}  \;   \textit{l}^{-1}.
\end{equation}
Therefore, $\tilde{\xi} \; := \; Image (\mathbf{p})$ is a smooth sub-bundle of the bundle $\xi \oplus \eta$.
\begin{proposition}
\par
-i) $\textit{l}: \; \xi \longrightarrow \tilde{\xi}$  is a bundle isomorphism
\par
-ii) away from the neighbourhood $2U$, the bundles $\tilde{\xi}$,  $\eta$ coincide  
\par
-iii) the triples $(\; \tilde{\xi}, \; \eta, \; \mathbf{I}_{\eta} \; )$, \;
$(\; \xi, \;\eta, \; a)$ are isomorphic and hence they define the same element of $K^{0}(T^{\ast}(M)).$
\end{proposition}

{\bf Proof.}
-i)  Obviously,
\begin{equation}
\textit{l}^{-1}(\; \tilde{\xi}  \;) = \textit{l}^{-1}(\; Image( \;
\textit{l} \;  \mathbf{p}_{1}  \;   \textit{l}^{-1}
    \;)  \;) 
    = \textit{l}^{-1}(\; Image( \;
\textit{l} \;  \mathbf{p}_{1}    
    \;)  \;) 
    = \; Image( \; \mathbf{p}_{1}    
    \;) = \xi .
\end{equation}
\par
-ii)
Let the subscript $_{\infty}$ denote the behaviour on the complement of $2U$. Then

%%%%%%%%%%%%%%%%%%%%%%%%%%%%%%%%%%%%%%%%%%
\begin{equation}
\textit{l}_{\infty} =
\left( \begin{array}{cc}
0&-\mathbf{a}^{-1}\\
a&0
\end{array}
\right), \hspace{1cm}  
\textit{l}_{\infty}^{-1} =
\left( \begin{array}{cc}
0&\mathbf{a}^{-1}\\
-\mathbf{a}&0
\end{array}
\right)
\end{equation}

%%%%%%%%%%%%%%%%%%%%%%%%%%%%%%%%%%%%%%

and therefore

\begin{equation}
\tilde{\xi}_{\infty} = Image ( \;
\textit{l}_{\infty} \;  \mathbf{p}_{1}  \;   \textit{l}_{\infty}^{-1}
    \;)   
\end{equation}

\begin{equation}
= Image \left  
 ( \begin{array}{cc}
0&-\mathbf{a}^{-1}\\
\mathbf{a}&0
\end{array}
\right) \;
\left( \begin{array}{cc}
1&0\\
0&0
\end{array}
\right) \;
\left( \begin{array}{cc}
0&\mathbf{a}^{-1}\\
-\mathbf{a}&0
\end{array}
\right)
\end{equation}

\begin{equation}
= Image \; \left( \begin{array}{cc}
0&0\\
0&1
\end{array}
\right)  \; = \; \eta_{\infty}
\end{equation}

%%%%%%%%%%%%%%%%%%%%%%%%%%%%%%%%%%%%%%%%%%

\par
-iii) It is sufficient to verify the commutativity of the diagram on the complement of $2U$
\begin{equation}
\begin{array}{ccc}
               \xi_{\infty}  &    \stackrel{a}{\longrightarrow}  & \eta_{\infty} \\              	
                 \downarrow  &                                   & \downarrow \\
         \tilde{\xi}_{\infty}=\eta_{\infty}& \stackrel{Id}{\longrightarrow} & \eta_{\infty}
\end{array}
\end{equation}
where the first vertical arrow is $\textit{l}_{\infty} $ and the second vertical arrow is the identity.
In fact, this is true because 

\begin{equation}
\textit{l}_{\infty}  \; \textit{p}_{1} = \left( \begin{array}{cc}
0&0\\
a&0
\end{array}
\right) 
\end{equation}

\par
\begin{corollary}
-i) The operator $\textbf{R} = \textbf{P} - \textbf{P}_{2}$ represents the quantization of the quasi-local
residue bundle $\mathbf{r}$.
\par
-ii) Passing to tne Chern character, one has
\begin{equation}
Ch(\; \xi, \;\eta, \; a) \; = \;Ch (\; \tilde{\xi}, \; \eta, \; \mathbf{I}_{\eta} \; ) \; \in H^{ev}_{comp} (T^{\ast}(M))
\end{equation}
\par
-iii) For any two linear or direct \cite{Teleman2} , \cite{Teleman3} connections $\nabla_{\tilde{\xi}}$, resp. $\nabla_{\eta}$, on the bundle $\tilde{\xi}$, resp. on
$\eta$,  which coincide on the complement of $2U,$  one has
\begin{equation}
Ch(\; \xi, \;\eta, \; \mathbf{a}) \; = \; Ch(\; \tilde{\xi}, \nabla_{\tilde{\xi}} \;) \;- \;  Ch(\; \eta, \nabla_{\eta} \;) 
\end{equation}
-iv) $Ch(\; \xi, \;\eta, \; \mathbf{a}) $ depends only on the the bundle homomorphism $\mathbf{r}$, which has support on $2U$.
\end{corollary}
\par
 The bundle $\textbf{r}$ prepares the triple $(\; \xi, \;\eta, \; a)$ for computing its Chern character via Chern-Weil theory.
\par
By stabilization of the symbol $\sigma \; = \;(\; \xi, \;\eta, \; a) $  we may assume that the bundles $F$ and $\eta$ are trivial. Then
\begin{equation}
Ch ( \;\sigma \; ) = Ch (\; \tilde{\xi}) - Rank (\xi).
\end{equation} 
\par
If $\mathbf{r}(x, y)$ is a direct connection \cite{Teleman2} , \cite{Teleman3} on the bundle $\tilde{\xi}$, we may assume
that it is flat on the complement of $2U$.
Using the results of \cite{Teleman2}, \cite{Teleman3}, \cite{Kubarski-Teleman}, we obtain the
\begin{theorem} 
The components of the Chern character of  $\sigma $, seen as a periodic cyclic
homology classes, are 
\begin{equation}
Ch_{q} ( \; \sigma \; ) =  (-1)^{dim M}\frac{ (2\pi i)^{q}   (2q)! }{ q !  }
. \; Tr\hspace{0.1cm}[ \mathbf{r}(x_{0},x_{1})\mathbf{r}
(x_{1},x_{2})...\mathbf{r}(x_{k-1},x_{q})\mathbf{r}(x_{q},x_{0})],
\end{equation}
for $q$ even number.
\end{theorem}
\par
%%%%%%%%%
\section{Review of Hochschild and Cyclic Homology.}       %%%%%%%  sect. 6
In this section we review results due to Connes  \cite{Connes1}, \cite{Connes2}  and Connes-Moscovici \cite{Connes-Moscovici1}.
\par
\subsection{Hochschild and Cyclic Homology.}  %%%%%%%%%% Sect. 6.1
\par
Let $\mathit{A}$ be an associative algebra with unit over the field $\mathbb{K} = \mathbb{R}$ or $\mathbb{C}$.
\par
Suppose $\mathit{A}$  is a left and right module over $\mathit{L}$, which is a unitary ring over $\mathbb{K}$.
\par
Define for any $k = 0,1, 2,...$
\begin{equation}
C_{\mathit{L}, k} (\mathit{A}) := \otimes_{\mathit{L}}^{k+1} \mathit{A}.
\end{equation}
%%%%%%%%%%%%%
We have assumed here that the tensor product $\otimes_{\mathit{L}}^{k+1} $ is circular, i.e., for any $f_{0},f_{1},...,f_{q} \in \mathit{A}$ and $f \in \mathit{L}$, one has 
\begin{equation}
f_{0} \otimes_{\mathit{L}} f_{1} \otimes_{\mathit{L}} ,...,\otimes_{\mathit{L}} f_{k} . f \otimes_{\mathit{L}}  =
f. f_{0} \otimes_{\mathit{L}} f_{1} \otimes_{\mathit{L}} ,...,\otimes_{\mathit{L}} f_{k}  \otimes_{\mathit{L}},  
\end{equation}
see \cite{Loday}, Sect. 1.2.11.
%%%%%%%%%%%%%
In particular,
\begin{equation}
C_{\mathit{L}, 0} (\mathit{A}) := \otimes_{\mathit{L}}^{1} \mathit{A} = \frac{\mathit{A}}{[\mathit{A}, \mathit{L}]}.
\end{equation}
For negative integers $k$ one defines  $C_{\mathit{L}, k} (\mathit{A}) := 0$.
\par
The bar operator $b'_{k}: C_{\mathit{L}, k} (\mathit{A})  \longrightarrow   C_{\mathit{L}, k-1} (\mathit{A})$ 
is defined by
\[
b'_{k} (f_{0} \otimes_{\mathit{L}} f_{1} \otimes_{\mathit{L}} ,...,\otimes_{\mathit{L}} f_{k} \otimes_{\mathit{L}}) :=
\]
\begin{equation}
:= \sum_{r=0}^{r=k-1} (-1)^{r}f_{0} \otimes_{\mathit{L}} f_{1} \otimes_{\mathit{L}} ,...,\otimes_{\mathit{L}}(f_{r}.f_{r+1}) \otimes_{\mathit{L}} ,...., \otimes_{\mathit{L}} f_{k} \otimes_{\mathit{L}}.
\end{equation}
In particular, $b'_{0} = 0$.
\par
The Hochschild operator $b_{k}: C_{\mathit{L}, k} (\mathit{A})  \longrightarrow   C_{\mathit{L}, k-1} (\mathit{A})$ is defined by
\[
b_{k} (f_{0} \otimes_{\mathit{L}} f_{1} \otimes_{\mathit{L}} ,...,\otimes_{\mathit{L}} f_{k} \otimes_{\mathit{L}}) :=
\]
\begin{equation}
:= b'_{k} (f_{0} \otimes_{\mathit{L}} f_{1} \otimes_{\mathit{L}} ,...,\otimes_{\mathit{L}} f_{k} \otimes_{\mathit{L}})  \\
+(-1)^{k}(f_{k}.f_{0}) \otimes_{\mathit{L}} f_{1} \otimes_{\mathit{L}} ,...,\otimes_{\mathit{L}} f_{k-1} \otimes_{\mathit{L}}) .
\end{equation}
\par
Both operators $b'$ and $b$ satisfy $(b')^{2} = b^{2} = 0$. The corresponding complex
 $\{ C_{\mathit{L}, \ast}(\mathit{A}), b \}$ is called Hochschild complex of the algebra  $\mathit{A}$ over the ground ring $\mathit{L}$, see \cite{Loday}, 1.2.11. The homology of this complex is called \emph{Hochschild homology of the algebra} $\mathit{A}$ \emph{over the ground ring} $\mathit{L}$ and is denoted by $HH_{\mathit{L}, \ast}(\mathit{A})$. 
\par
The $\emph{cyclic permutation}$  $T: C_{\mathit{L}, k}(\mathit{A})  \longrightarrow C_{\mathit{L}, k}(\mathit{A}) $
is defined by
\[
T (f_{0} \otimes_{\mathit{L}} f_{1} \otimes_{\mathit{L}} ,...,\otimes_{\mathit{L}} f_{k} \otimes_{\mathit{L}}) :=
\]
\begin{equation}
:= (-1)^{k }f_{1} \otimes_{\mathit{L}} f_{2} \otimes_{\mathit{L}} ,...,\otimes_{\mathit{L}} f_{k} \otimes_{\mathit{L}}
f_{0} \otimes_{\mathit{L}}. 
\end{equation}
An element of $ f \in C_{\mathit{L}, k}(\mathit{A})$ is called $\emph{cyclic}$  if $T(f) = f$.
\par
If the algebra $\mathit{A}$ has unit, the corresponding bar complex is acyclic. However, the cyclic elements 
form the  $\emph{cyclic sub-complex}$ of the bar complex
\begin{equation}
C_{\mathit{L}, k} ^{\lambda}(\mathit{A}) := \{  f |  f  \in  C_{\mathit{L}, k} ^{\lambda}(\mathit{A}) , \hspace{0.2cm} T(f) = f\}
\end{equation}
 and its homology is interesting. Its homology is called \emph{cyclic homology} of the algebra $\mathit{A}$ over the ring $\mathit{L}$, and it is denoted $H_{\mathit{L}, k} ^{\lambda}(\mathit{A})$.
 \par
 If the ring $\mathit{L}$ coincides with the ground field $\mathbb{K}$, then $\mathit{L}$ is omitted from the notation
 and the corresponding homology is called \emph{cyclic homology} of $\mathit{A}$.
 %%%%%%%%%%%%% Begin Connes Long Exact Sequence 
\begin{theorem} ( Connes' $I.S.B.$ Long Exact Sequence, see \cite{Connes1} Part II, Sect. 4, p. 119. , \cite{Loday} Sect.2.1.4.
and in co-homological context \cite{Connes2} Pg. 205). 
\par
There exist functorial homomorphisms $I$, $B$ and $S$ such that the following long sequence
\begin{equation}    %%%%%%%% Eq.  (49)
... \stackrel{B}\longrightarrow HH_{\mathit{L}, k} (\mathit{A})   \stackrel{I}{\longrightarrow}
    H^{\lambda}_{\mathit{L}, k} (\mathit{A})   \stackrel{S}{\longrightarrow}
   HH^{\lambda}_{\mathit{L}, k-2}  (\mathit{A})  \stackrel{B}{\longrightarrow}
   HH_{\mathit{L}, k-1}  (\mathit{A}) \stackrel{I}{\longrightarrow} ....
\end{equation}
is exact.
\end{theorem}  
 %%%%%%%%%%%%%  End Connes Long Exact Sequence
\par
Recall $\mathit{M}_{n}(\mathit{A})$ denotes the algebra of $nxn$ matrices with entries in $\mathit{A}$.
\par
%%%%%%%%%%%%%%%%%%   BEGIN   MORITA   Theorem  6
\begin{theorem} (Morita invariance Theorem) (see e.g. \cite{Loday} Sect. 1.2.5-7  , \cite{Cuntz} Sect. 2.8)
\par
Replacing the algebra $\mathit{A}$ with its matrix algebra  $\mathit{M}_{n}(\mathit{A})$ does not change its Hochshild and cyclic homology.
\end{theorem}
\begin{definition}
%%%%%%%%%%%%%%%%%%%% END  MORITA
A unital  $\mathit{L}$-algebra is called \emph{separable} over  $\mathbb{K}$ 
if the multiplication mapping $\mu: \mathit{L} \otimes_{\mathbb{K}}{\mathit{L}}^{op}  \longrightarrow \mathit{L}$ 
has a $\mathit{L}$-bimodule splitting, see \cite{Loday} Sect. 1.2.12.
\end{definition}
\par
\begin{lemma}
Let $e \in \mathit{A}$  be an idempotent ($e^{2} = e$) and
$\Lambda  := \mathbb{K} + \mathbb{K} e $.
\par
Then $\Lambda$ is separable over $\mathbb{K}$.
\end{lemma}
\begin{proof}
 Indeed, the splitting
$s: \Lambda \longrightarrow \Lambda \otimes_{\textit{K}}\Lambda^{op}$
is defined on the generators $1$ and $e$  by
\[
s(\;1\;) = e \otimes_{\textit{K}} e + (1-e) \otimes_{\textit{K}} (1-e)
\]
\begin{equation}
s(\;e\;) = e \otimes_{\textit{K}} e.
\end{equation}
\end{proof}
\par
\begin{theorem} (see \cite{Loday} Theorem 1.2.12-13.)  %%%%  Theorem 9
\par
Let $\Lambda$ be a separable algebra over $\textit{K}$ and $\mathcal{U}$ be a unital $\Lambda$-algebra.
\par
Then there is a canonical isomorphism
\begin{equation}
HH_{k} (\; \mathcal{U}  \;)\; \cong \; HH_{\Lambda, k} (\; \mathcal{U}  \;).
\end{equation}
\par
If, in addition, $\Lambda$ is a subalgebra of $\mathcal{U}$, then the canonical
epimorphism
\begin{equation}
\phi_{\Lambda}:  C_{k}(\;\mathcal{U}\;) \longrightarrow  C_{\Lambda, k}(\;\mathcal{U}\;)
\end{equation}
induces isomorphisms in homology.
\end{theorem}
\par
\begin{corollary}   %%%%%%   Corollary 10
The canonical epimorphism
\begin{equation}
\phi:  C_{k}(\;\mathcal{U}\;) \longrightarrow  C_{\Lambda, k}(\;\mathcal{U}\;)  %%%% (53)
\end{equation}
induces isomorphisms
\begin{equation}
\phi_{k\ast}: HH_{k} (\; \mathcal{U}  \;)\; \cong \; HH_{\Lambda, k} (\; \mathcal{U}  \;)  %%%  (54)
\end{equation}
\begin{equation}
\phi_{k\ast}^{\lambda}: H_{\Lambda, k}^{\lambda} (\; \mathcal{U}  \;)\; \cong 
\; H_{\Lambda, k}^{\lambda} (\; \mathcal{U}  \;).   %%%%%   (55)
\end{equation}
\end{corollary}
\par
If the algebra $\mathit{A}$ is a Fr$\acute{e}$ch$\acute{e}$t algebra, then the algebraic tensor products are usually replaced with \emph{projective} tensor products, see Connes \cite{Connes1}, Part II, Sect. 6. The homologies of the projective tensor product completions are called, respectively, \emph{continuous} Hochschild and cyclic homologies; in these cases, the adjective \emph{continuous} is tacitly understood.  
\par 
The following result constitutes the basic link between the Hochschild and cyclic homology, on a one side, and the classical differential forms $\Omega^{\ast}(M) $ and the \emph{de} Rham (Alexander-Spanier) cohomology, on the other side.
\begin{theorem} Connes \cite{Connes1}, \cite{Connes2}, Teleman \cite{Teleman1}
For the Fr$\grave{e}$ch$\grave{e}$t algebra $C^{\infty}((M))$ on any paracompact smooth manifold $M$
\begin{equation}
-i)   \hspace{1cm} HH_{k} C^{\infty}((M)) \;=\; \Omega^{k} (M)
\end{equation}
\begin{equation}
-ii) \hspace{0.5cm}
H^{\lambda}_{k}(C^{\infty}(M)) = \frac{\Omega^{k}(M)}{d\Omega^{k}(M)} \oplus H_{dR}^{k-2}(M) \oplus,...,
\oplus H_{dR}^{\epsilon}(M),   \hspace{0,2cm} \epsilon = 0, \; 1
\end{equation}
where $\epsilon$ has  the parity of $k$.
\end{theorem}
Part -i) of this theorem is due to Connes \cite{Connes1} Part II, Lemma 45, for $M$ compact, in the context of Hochschild \emph{co}-homology. For paracompact manifolds, the result is due to Teleman \cite{Teleman1}, in the context of Hochschild homology.
\par  
Part -ii) of the theorem is a formal consequence of -i) combined with Connes'  $(I.S.B)$-exact sequence, Theorem 5.

\subsection{Chern Character of Idempotents.}      %%%%%%%  Sect. 6.2

In this section we consider the Hochschild and cyclic homology of the arbitrary associative algebra with unit $\mathit{A}$ ( the ring $\mathit{L} = \mathbb{K}$ ), see the previous section.
\begin{proposition}     %%%%%%%   Proposition 12
Suppose $p \in \mathit{M}_{n}(\mathit{A})$ is an idempotent, $ p^{2} = p$.
Then for any even number $q$, the chain 
\begin{equation}   %%%%%% To verify coefficient
\Psi_{q} (p) :=   \frac{(2\pi i)^{q}q !}{(q/2) !}   p \otimes_{C} p \otimes_{C} , ... , \otimes_{C} p   \hspace{0.3cm}  ( q+1 \;\; factors)
\end{equation}
is a cyclic cycle with cyclic homology class 
\begin{equation}
[\Psi_{q}(p)] \in H^{\lambda}_{q} (\mathit{M}_{n} (\mathit{A})) \equiv 
H^{\lambda}_{q}  (\mathit{A}).
\end{equation}
\end{proposition}
\begin{definition}
The system of cyclic homology classes 
\begin{equation}
Ch (p) := \{  [\Psi_{q}(p)]  \}_{q = ev}
\end{equation}
 is called the Chern character of the idempotent $p$. 
\end{definition}

\subsection{Pairing of Cyclic Homology of Algebras of Operators with Alexander-Spanier Co-homology}    %%   Sect. 6.3
\par 
Let $Op$ denote some algebra of bounded operators in the Hilbert space $H$ of $L_{2}$ sections in a vector bundle $\xi$ over the smooth compact manifold $M$. In our applications, $Op$ will be one Schatten ideal of operators on $H$.
\begin{definition}  %%%%%%%%%  Definition 14
Let
\begin{equation}
\phi_{k} = \sum A_{0} \otimes_{\mathbb{C}} A_{1}\otimes_{\mathbb{C}} , ...,\otimes_{\mathbb{C}} A_{k} \in C^{\lambda}_{k} (Op)
\end{equation}
and 
\begin{equation}
\eta^{k} = \sum f_{0}  \otimes_{\mathbb{C}}  f_{1} \otimes_{\mathbb{C}} , ..., \otimes_{\mathbb{C}}  f_{k} , \hspace{0.3cm}
f_{i} \in C^{\infty} (M),
\end{equation}
be convergent series in the projective tensor product spaces. One  assumes also that in each monomial at least one of the  factors $A_{i}$ is trace class.
\par
The pairing $\phi \;\cap \; \eta$, called \emph{cap product} is defined  on chains by the convergent double series
\begin{equation}
\phi_{k} \;\cap\; \eta^{k}  := \sum \sum Tr (  A_{0} f_{0}  A_{1} f_{1} , ... ,  A_{k} f_{k}  ) \in \mathbb{C},           
\end{equation}
where $Tr$ denotes operator trace.
\end{definition}
\par
In this definition both the cyclic homology class and the Alexander-Spanier co-homology class should have the same degree. 
For this purpose, we introduce 
\begin{equation}
H^{\lambda}_{ev} (Op) \;:= \; \sum_{q=even} H^{\lambda}_{q} (Op)
\end{equation}
and
 \begin{equation}
 H^{ev}_{AS} (M) \;=\; \sum _{q=even} H^{q}_{AS} (M)
 \end{equation}
 \par
 \begin{definition}   %%%%%%   Definition 15
 \begin{equation}
 H^{\lambda}_{ev} (Op) \; \Diamond \; H^{ev}_{AS} (M) \;=\; \{ \; \sum \phi_{q} \;\otimes \; \eta^{q}   \;  | \; 
  \{  \phi_{q} \;\otimes \; \eta^{q}   \}_{q=0, 2,4, ...} \in    
 H^{\lambda}_{ev} (Op)  \otimes H^{ev}_{AS} (M)  \} 
 \end{equation}
 \end{definition} 
 \par
 It is important to notice that the outcome of the  $\Diamond$ product consists of systems of tensor products of homology by co-homology classes rather than scalars.  In addition, one has the relation
 \[
Tr (  \phi \; \Diamond \; \eta   ) \; = \; \phi \; \cap \; \eta. 
\] 
\par
\begin{lemma} (Connes-Moscovici \cite{Connes-Moscovici1}, Lemma 2.1 (ii))   %%%%  Lemma 16
\par
Under the hypotheses of the above lemma, one has
\begin{equation}
b'(\phi_{k+1}) \;\cap\; \eta^{k}  :=  \phi_{k+1} \;\cap\; (\delta \eta^{k}), 
\end{equation} 
where $\eta$ is anti-symmetric Alexander-Spanier $k$-cochain  and $\delta$ denotes the Alexander-Spanier co-boundary.
\end{lemma}

\begin{corollary} 
The pairing $\cap$ passes to homology
\begin{equation}
 \cap: H^{\lambda}_{k} (Op) \otimes_{\mathbb{C}} H^{k}_{AS} (M) \longrightarrow \mathbb{C}
\end{equation}
\end{corollary}

\begin{corollary}
The cap product $\cap$ being linear in each factor, if one fixes one of the factors and make variable the other, one 
obtains an element of its dual. In particular, if $\phi_{k} $ is a cyclic \emph{co-cycle}, then the correspondence
\begin{equation}
\eta^{k}  \to \phi_{k} \;\cap\;  \eta^{k}  
\end{equation}
induces a mapping
\begin{equation}
      H_{AS}^{k} (M) \to \mathbb{C}
\end{equation}
and therefore it defines an Alexander-Spanier \emph{homology} class on $M$.
\end{corollary}
\par
%%%%%%%%%%%%%%%%   Sect. 7 
\section{\emph{Local} Cyclic Homology and \emph{Local} $K$-Theory of Schatten Ideals.}
\begin{remark}
It is clear that none of the Schatten ideals $\mathcal{L}^{p}$ of compact operators in the Hilbert space of $L_{2}$ sections in the fibre bundle $\xi$ over the manifold $M$ carries any kind of information about the space $M$ because all such Schatten ideals are \emph{independent} of the space $M$. To restore information about the base space $M$, it would be necessary to take into account the module structure of these spaces over some algebra of functions over $M$.
A partial solution to this problem is to consider the distributional kernel of such operators and keep track of their \emph{supports}, as the Alexander-Spanier co-homology does.
\end{remark}
\par
The leading idea of this paper is to introduce, in analogy with the Alexander-Spanier co-homology, the
\emph{local} $K$-theory ( $K^{i, loc}(A)$ ), \emph{local} Hochschild homology and \emph{local} cyclic homology ($HH_{\ast}^{loc}$, $H_{\ast}^{\lambda, loc}(A)$), etc. We believe that these new structures fit more naturally (than the classical counter-parts) in many interesting problems which involve Alexander-Spanier homology and co-homology, including the local index theorem. 
\par
The basic ideas of this program were announced at the ''Trieste - 2007 Workshop on Noncommutative Geometry'', October 2007.
\par
We stress that our \emph{local} cyclic homology definitely differs from the local cyclic homology introduced by Puschnigg \cite{Puschnigg}.   Here, the word  \emph{local} refers to the supports of the operators involved. 
\par
The first algebras of interest onto which we intend to apply our \emph{local} structures  are the Schatten ideals $\mathcal{L}^{p}$,  $1 \leq p \leq \infty$. It is well known and it is not surprising that the $K$-theory and cyclic homology of Banach algebras and, in particular, of these algebras are trivial or do not describe well the topology of the spaces onto which they are defined, see \cite{Connes2} , \cite{Connes-Moscovici1} Sect. 4, pg.371,    \cite{Cuntz} Corollary 3.6 . 
\par
Our arguments are independent of the notion of \emph{entire} cyclic co-homology or \emph{asymptotic} co-homology, due to Connes (see \cite{Connes2} for more information)  and Connes-Moscovici \cite{Connes-Moscovici1}.   
\par
To proceed, for any element $A$ belonging to the Schatten ideal $\mathcal{L}^{p}$ of operators on the Hilbert space of $L_{2}$ sections in a complex vector bundle over the compact smooth manifold $M$ we consider its support $Supp(A) \subset M\times M$ given by the support of its distributional kernel. 
\par
The support  filtration of $Supp (A) \subset M \times M$ gives a filtration of the elements of the Schatten ideal. 
\par
The support-filtration on each of the factors of the  $\ \gamma = \sum A_{0} \otimes_{\mathbb{C}} A_{1} \otimes_{\mathbb{C}} 
,..., \otimes_{\mathbb{C}} A_{k} \in  C^{\lambda}_{k} (\mathcal{L}^{p})$ gives a filtration in $C^{\lambda}_{k} (\mathcal{L}^{p})$ in the same manner in which Alexander-Spanier co-homology is defined. To be more specific, 
we will say that $\gamma$ has support in $\mathit{U} \subset M \times M$ iff $Supp (A_{i}) \subset   \mathit{U} $ for any $i$. 
\par
By definition, $C^{\lambda, \mathit{U}}_{k} (\mathcal{L}^{p})$ consists of those chains $\gamma$ whose supports lay in  $\mathit{U}$.
\par
By analogy with the definition of the Alexander-Spanier co-homology,
\begin{equation}
H^{\lambda, loc}_{\ast} (\mathcal{L}^{p})
\end{equation}  
is the homology of the sub-complex $C^{\lambda, \mathit{U}}_{\ast} (\mathcal{L}^{p})$ consisting of those chains $\gamma$ whose supports are contained in the subset $\mathit{U} \subset M \times M$, with $\mathit{U}$ sufficiently small.
\par
The $K$-theory groups $K^{loc, i} (\mathcal{L}^{p})$ are defined analogously.
\par
These groups will be called \emph{local} cyclic homology, resp. \emph{local} $K$-theory. 
\par
An extension of these constructions to more general Banach algebras will be discussed elsewhere.
\par
We will show in Sect. 9, Proposition 25 and Propsition 26, that the \emph{local} cyclic homology and the real \emph{local} $K$-theory of the algebra of smoothing operators is at least as big as the de Rham cohomology and ordinary real $K$-theory, respectively.  
%%%%%%%%%%%   End  Sect. 7
%%%%%%%%%%%%%%%%%%%%%%%%%%%%%%%%%%%%%%%%%%%%%
\section{Applications of the \emph{Local} Cyclic Homology }   %%%%%% Sect 8
\par
In this section we show that the Connes-Moscovici local index Theorem 1 may be reformulated by using 
our notions of \emph{local} structures introduced in Sect. 7. 
\par
In  Sect. 4  we summarised the Connes-Moscovici construction of the index class. It is obtained as the result of the composition of two constructions
\begin{equation}
K^{1}(C(S^{\ast}M)) \stackrel{\mathbf{R}}{\longrightarrow}  K^{0} ({\mathcal{K}}_{M}) \stackrel{\tau}{\longrightarrow}  H_{ev}^{AS} (M),
\end{equation}
where $K^{0} ({\mathcal{K}}_{M}) $ consists of differences of stably homotopy classes of smooth idempotents with \emph{arbitrary} supports.
\par
In \cite{Connes-Moscovici1}, Pg. 352 it is clearly stated that the connecting homomorphism $\partial_{1}$ takes values in $ K^{0} ({\mathcal{K}}_{M}) \simeq \mathbb{Z}$ and that the information carried by it is solely the index of the operator. However, \cite{Connes-Moscovici1} states also that by pairing the residue operator $\mathbf{R}$ with the Alexander-Spanier cohomology (which is \emph{local}) one recovers the whole co-homological information carried by the symbol, as stated by Theorem 1 of Sect. 4. 
\par
Let $\mathit{S}$ denote the unitarized algebra of smoothing operators on the compact smooth manifold $M$
\begin{equation}
\mathit{S} := \mathbb{C}.1 + \mathcal{L}^{\infty}(M).
\end{equation}
All results of this section will remain valid if we replace smoothing operators by trace class operators.
\par
The leading ideas of the paper are: -1)  to consider $\mathbf{R}$ not only as element of $K^{0} (\mathit{S})$, but rather as an element of $K^{0, loc} (\mathit{S})$, -2)  show that $\mathit{R}$ has a genuine Chern character belonging to the \emph{local} cyclic homology $H^{\lambda, loc, ev}_{\Lambda} (\mathit{M}_{2} (\mathit{S}))$  over the separable ring
$\Lambda := \mathbb{C} + \mathbb{C} e $, where $e = \mathbf{P}_{2}$ is the trivial idempotent involved in the construction of
the residue operator $\mathit{R}$, see formula (20) and Lemma 8, Sect. 6.1.
%%%%%%%%%%%%%%%%%%%%%%%%%%%%%%%%%%  Subsect. 8.1

\subsection{\emph{Local} Chern Character of the Residue Operator $\mathbf{R}$}
 Let the ring $\mathit{L}$  from Section  5.1. be given by 
\begin{equation}
\Lambda = \mathbb{C}  +  \mathbb{C} e 
\end{equation}
\par
We assume that $\mathcal{U}$ is quasi-stable under products and stable under products by elements of $ \Lambda$;  (here, by quasi-stable we intend a property analogous to that refering to products of compact operators). This is the space of operators with \emph{small} supports. We apply next the considerations made in Sect. 5. to produce \emph{local} cyclic homology $H^{\lambda, loc}_{\Lambda, \ast} (\mathit{S})$; here, $\mathit{S}) := \mathcal{S} + \mathbb{C}.1$.
\par
\begin{remark}
Let $\mathbf{R} =  \mathbf{P}   -  \mathbf{e}$, where $\mathbf{P}$ and  $ \mathbf{e}$ are idempotents.
Then $\mathbf{R}$  satisfies the identity
\begin{equation}
\mathbf{R}^{2} =  \mathbf{R} - (\mathbf{e} \mathbf{R} + \mathbf{R} \mathbf{e}).            
\end{equation}
Notice that if $R$ were an idempotent, the  term $\mathbf{e} \mathbf{R} + \mathbf{R} \mathbf{e}$ would be absent.
\end{remark}

\begin{theorem}
Let  $\mathbf{P}, \mathbf{e} $  be idempotents in   $\mathcal{U}$ and $ \mathbf{R}= \mathbf{P} - \mathbf{e}$.
\par
Then, for any even number $q$
\begin{equation}
\tau_{q}(\mathbf{R}) := R\otimes_{ \Lambda}\mathbf{R} \otimes_{ \Lambda} ..... \otimes_{ \Lambda} \mathbf{R} \; \otimes_{ \Lambda} \in C^{\Lambda}_{q}(\mathcal{U})
\end{equation}
is a local cyclic cycle of the algebra $\mathit{S}$ over the ring $\Lambda$.
\end{theorem}
\par
\noindent
\begin{proof}
It is clear that $\tau_{q}(\mathbf{R})$ is cyclic. 
\par
To simplify the notation, limited to this proof, we write $ \otimes = \otimes_{ \Lambda}$ and we omit
the last tensor product. We have

\begin{equation}
b'(\mathbf{R} \otimes \mathbf{R} \otimes ..... \otimes \mathbf{R} ) = 
\sum_{r=0}^{r=q-1}
(-1)^{r} \mathbf{R} \otimes \mathbf{R} \otimes ...\otimes \mathbf{R}^{2} \otimes ... \otimes
\mathbf{R}
\end{equation}
%%%%%%%%%%%%%%%%%%
\begin{equation}
= \sum_{r=0}^{r=q-1} (-1)^{r} \mathbf{R} \otimes ...\otimes \mathbf{R} \otimes 
[\mathbf{R} - (\mathbf{e} \mathbf{R} + \mathbf{R} \mathbf{e})]
\otimes \mathbf{R} \otimes... \otimes \mathbf{R}  
\end{equation}
%%%%%%%%%%%%%%%%%%
\begin{equation}
= \sum_{r=0}^{r=q-1} (-1)^{r} \mathbf{R} \otimes ...\otimes \mathbf{R} \otimes 
\mathbf{R}
\otimes \mathbf{R} \otimes... \otimes \mathbf{R}  
\end{equation}
\begin{equation}
- \sum_{r=0}^{r=q-1} (-1)^{r} \mathbf{R} \otimes ...\otimes \mathbf{R} \otimes 
(\mathbf{e} \mathbf{R} + \mathbf{R} \mathbf{e})
\otimes \mathbf{R} \otimes... \otimes \mathbf{R}  
\end{equation}
%%%%%%%%%%%%%%%%%%
\begin{equation}
= - \sum_{r=0}^{r=q-1} (-1)^{r} \mathbf{R} \otimes ...\otimes \mathbf{R} \otimes 
(\mathbf{e} \mathbf{R} + \mathbf{R} \mathbf{e})
\otimes \mathbf{R} \otimes... \otimes \mathbf{R}  
\end{equation}
%%%%%%%%%%%%%%%%%%
\begin{equation}
= - \{ 
\sum_{r=0}^{r=q-1} (-1)^{r} \mathbf{R} \otimes ...\otimes \mathbf{R} \otimes 
\mathbf{e} \mathbf{R} 
\otimes \mathbf{R} \otimes... \otimes \mathbf{R} 
\end{equation}
\begin{equation}
 \hspace{0.3cm}+ \hspace{0.4cm} \sum_{r=0}^{r=q-1} (-1)^{r} \mathbf{R} \otimes ...\otimes \mathbf{R} \otimes 
 \mathbf{R} 
\otimes \mathbf{e}\mathbf{R} \otimes... \otimes \mathbf{R}
\} 
\end{equation}
%%%%%%%%%%%%%%%%%%

\begin{equation}
= - \{ 
 \; \mathbf{e} \mathbf{R} \otimes ...\otimes \mathbf{R} \otimes 
 \mathbf{R} 
\otimes \mathbf{R} \otimes... \otimes \mathbf{R} 
\end{equation}

\begin{equation}
 \hspace{0.3cm} + (-1)^{q-1} \mathbf{R} \otimes ...\otimes \mathbf{R} \otimes 
 \mathbf{R} 
\otimes \mathbf{R} \otimes... \otimes \mathbf{R} \mathbf{e} \;
\} = 0,
\end{equation}
which completes the proof.
\end{proof}

\begin{definition}
The Chern character of the residue operator $\mathbf{R}$ is the system of \emph{local} cyclic homology classes of the algebra $\mathit{S}$ over the ring $\Lambda$
\begin{equation}
Ch( \mathbf{R} ) := \{    [ \;  Ch_{q} (\mathbf{R}) \; ] \}_{q=0, 1,...} \in H^{loc}_{\Lambda, ev}
\end{equation}
where
\begin{equation}
Ch_{q} (\mathbf{R}) := \frac{  (2\pi i )^{q}  (2q) !  } 
          { q !} \otimes_{\Lambda}^{2q+1} \mathbf{R}
\end{equation}
\end{definition}
\par
%%%%%%%%%%%%%%%%%%%%%%%%%%%%%%%%%%%%%%%%%%%%%%%%
\subsection{Pairing of $H_{\Lambda, \ast}^{\lambda, loc }$ with Alexander-Spanier Co-homology.} %%%  Sect. 8.2
Given that the multiplication of operators on $M$ by smooth functions on $M$ does not increase distributional supports of operators and that smooth functions on $M$ commute with $\Lambda$, the Connes-Moscovici pairing Lemma 16, Sect. 5.3. passes to the the factor spaces $C^{\lambda, loc}_{\Lambda, \ast}$. 
The same formula (63)  gives an induced pairing
\begin{equation}
\cap_{\Lambda, loc}: C^{\lambda, loc}_{\Lambda, k} \otimes_{\mathbb{C}} \Omega_{AS}^{k} (M) \longrightarrow 
\mathbb{C},
\end{equation}
where $\Omega_{AS}^{k}(M)$  denotes Alexander-Spanier $k$-cochains on $M$. It is understood that in order for the cup product to be defined it is required that both factors have the same degree. 
\par
 $Ch (\mathbf(R))$ being already a \emph{local} cyclic cycle of the algebra $\mathit{S}$ over the ring $\Lambda$, we get
the following interpretation of the formulas (24), (25)
\begin{equation}
Ind (\mathbf{a} \otimes_{\mathbb{C}} \phi ) = 
\frac{ q !} {  (2\pi i )^{q}  (2q) !  } 
          Ch_{q} (\mathbf{R}) \; \cap_{\Lambda, loc} \; \phi.
\end{equation}
Here, $\phi$ is any Alexander-Spanier  co-homology class of degree $2q$ represented by an anti-symmetric function defined on $M^{2q+1}$ with small support about the diagonal.
\par
In the context of \emph{local} cyclic homology, the formula (64) and Definition 15 of Sect. 6.3 become
%%%%%%%%%%%%%%%%%%%%%%%%%%%%%%%%%
\begin{equation}
H^{\lambda, loc}_{ev} (Op) \;:= \; \sum_{q=even} H^{\lambda, loc}_{q} (Op)
\end{equation}
and
 \par
 \begin{definition}
 \[
 H^{\lambda, loc}_{ev} (Op) \; \Diamond \; H^{ev}_{AS} (M) \;=
 \]
 \begin{equation}
 = \{ \; \phi_{q} \;\otimes \; \eta^{q}   \;  | \; 
  \{  \phi_{q} \;\cap \; \eta^{q}   \}_{q=0, 2,4, ...} \in    
 H^{\lambda, loc}_{ev} (Op)  \otimes H^{ev}_{AS} (M)  \} 
 \end{equation}
 \end{definition}
 \par
%%%%%%%%%%%%%%%%%%%%%%%%%%%%%%%%%%%%%%%%%%%%%%%%
\section{Connes-Moscovici Local Index Theorem.}   %%%    Sect. 9
The Connes-Moscovici Local Index Theorem (Connes-Moscovici \cite{Connes-Moscovici1} Theorem 3.9.) becomes
\begin{theorem} (Connes-Moscovici \emph{Local} Index Theorem)   %%%%%%%   Theorem 23
\par
Let $\mathbf{A}$ be an elliptic pseudo-differential operator on the smooth, compact manifold $M$ of even dimension and let $ [\phi] \in H^{2q}_{comp}(M)$. 
Let $\mathbf{a}$ be its symbol and let $\mathbf{R}_{\mathbf{a}}$ be its corresponding local residue smoothing operator. Then
\begin{equation}
Ch (\mathbf{R}_{\mathbf{a}}) \; \cap \;  [\phi] \;=\; (-1)^{dim M} <\;Ch \;\sigma (\mathbf{A}) \tau (M)  [\phi], \;[T^{\ast}M] \;>
\end{equation} 
where $\tau (M) = Todd (TM) \otimes C$ and $ H^{\ast} (T^{\ast}M)$ is seen as a module over $H^{\ast}_{comp}(M)$.
\end{theorem}
\par
We will show that this reformulation of the Connes-Moscovici local index theorem, combined with Poincar$\acute{e}$ duality, imply that $H^{\lambda, loc}_{\Lambda, \ast} (\mathit{S})$ contains 
$K^{0} _{comp }( (T^{\ast}M) ) \otimes \mathbb{R}$. 
\par
To show this, we introduce two homomorphisms
\[
\alpha: K^{1} ( C^{0}(S(T^{\ast}M) ) \longrightarrow H^{\lambda, loc}_{\Lambda, ev} (\mathit{S})
\]
\begin{equation}
\alpha ( \mathbf{a} ) \;: = \; Ch (\mathbf{R}_{\mathbf{a}} )
\end{equation}

\[
\beta: H^{\lambda, loc}_{\Lambda, ev} (\mathit{S}) \longrightarrow H^{AS}_{ev}(M)       
\]
\[
\beta (\phi) \in Hom_{\mathbb{C}} ( H^{ev}_{AS}(M), \mathbb{R} ) = H^{AS}_{ev} (M)
\]
\begin{equation}
(\beta (\phi)) (\eta) \;:=\; \phi \cap \eta \; \in \mathbb{C}.
\end{equation}
\par
\begin{proposition}    %%%%%%%%    Proposition  24
Let $M$ be a smooth, compact manifold $M$ of even dimension.
\par
Then the composition
\begin{equation}
 K^{1} ( C^{0}(S(T^{\ast}M) ) \otimes \mathbb{C} \stackrel{\alpha}{\longrightarrow} H^{\lambda, loc}_{\Lambda, ev} (\mathit{S})
 \stackrel{\beta}{\longrightarrow} H^{AS}_{ev} (M, \mathbb{C})
\end{equation}
is an isomorphism.
\end{proposition}
\begin{proof}
We have
\[ 
[(\beta \circ \alpha) (\mathbf{a}))] (\phi)  = [\;\beta \; (Ch (\mathbf{R}(\mathbf{a})) \;] ( \phi ) =
\]
(Theorem 21)
\begin{equation}
=\;\; (-1)^{dim M} <\;Ch \;\sigma (\mathbf{A}) \tau (M)  [\phi], \;[T^{\ast}M] \;>
\end{equation}
Suppose $\mathbf{a}$ is in the kernel of $\beta \circ \alpha$. This means that the RHS of (96) is zero for any even degree Alexander-Spanier cohomology class $\phi$. Poicar$\acute{e}$ duality implies that the co-homology class $Ch \;\sigma (\mathbf{A}) \tau (M) $ is zero. Given that the Todd class $\tau (M)$ is invertible, it follows that $Ch \;\sigma (\mathbf{A})$ is zero. As the Chern character is an isomorhism (over reals) from equivalence classes of virtual bundles to $H^{AS, ev}_{comp}(T^{\ast}M)$, it follows that  $\mathbf{a} \in K^{0} (T^{\ast}M)$ is the zero element. Therefore, $\beta \circ \alpha $ is a monomorphism.
\par
Given that $K^{0} (T^{\ast}M) \otimes \mathbb{R}$ and $H^{AS}_{ev} (M)$ have the same dimension, it follows that $\beta \circ \alpha$ is an isomorphism.
\end{proof}
\par
The same result extends to the Schatten ideal $\mathcal{L}^{1}$.
\begin{proposition}        %%%%%%%%%%%   Proposition 25
 The \emph{local} cyclic homology, resp.  the \emph{local} real $K$-theory, of the algebra of smoothing operators is at least as big as the de Rham cohomology, resp. ordinary real $K$-theory. 
 \par
 The same result extends to the Schatten ideal $\mathcal{L}^{1}$.
\end{proposition}
\begin{remark}   %%%%%%%%   Remark  26
The condition  $dim M = even$ in Theorem  23 may be dropped. In fact, our result follows from the Connes-Moscovici local index theorem, which is valid for manifolds of any dimension (see \cite{Connes-Moscovici1} Pg. 368).
\end{remark}

\begin{remark}   %%%%%%  Remark 27
Replacing the cyclic homology of the algebra of smoothing operators (which is trivial) by its \emph{local} cyclic homology, one recovers the possibility to control at least that homological information which makes the  index formula interesting.
\end{remark}
\par
Given that the local cyclic homology $H^{\lambda, loc }_{\Lambda, \ast}(\mathit{S}) $ is interesting, we are entitled to introduce a new pairing $\sqcap $, called \emph{square cap} product.
\begin{definition}   %%%%%%%%%  Definition 28 
The \emph{square} cup product
\begin{equation}
\sqcap: H^{\lambda, loc}_{\Lambda, k}(\mathit{S}) \otimes_{\mathbb{C}} H_{AS}^{k}(M) \longrightarrow H^{\lambda, loc}_{\Lambda, k}(\mathit{S})
\end{equation}
is defined on the elements
\begin{equation}
\phi_{k} = \sum A_{0} \otimes_{\mathbb{C}} A_{1}\otimes_{\mathbb{C}} , ...,\otimes_{\mathbb{C}} A_{k} \in C^{\lambda, loc}_{k} (S)
\end{equation}
\begin{equation}
\eta^{k} = \sum f_{0}  \otimes_{\mathbb{C}}  f_{1} \otimes_{\mathbb{C}} , ..., \otimes_{\mathbb{C}}  f_{k} , \hspace{0.3cm}
f_{i} \in C^{\infty} (M),
\end{equation}
 by the convergent double series
\begin{equation}
\phi_{k} \;\sqcap\; \eta^{k}  := \sum \sum  A_{0} f_{0}  A_{1} f_{1} , ... ,  A_{k} f_{k}   \in   C^{\lambda, loc}_{k} (S).       
\end{equation}
\end{definition}
\par
 It is important to notice that square cap product $\sqcap$ does not require to manufacture trace class operators. The outcome of the square cup product is not a system of scalars, but rather a system of tensor products of homology and co-homoogy classes.  The definition makes sense for $S$ replaced by more general operator algebras.
%%%%%%%%%%%%%%%%%%%%%%%%%%%%%%%%%%%%%%%%%%%%%%
\section{Further Extensions and Problems}  %%%%%%%%%%  Sect.   10
A more detailed study of the arguments presented in Sect.10.1-3  will be discussed elsewhere.
\par
\subsection{\emph{Local} Chern Character on Banach Algebras Extensions}  %%%%%%%%  Subsection 10.1
In analogy with the short exact sequence (19) of operator algebras considered in  Sect. 3, we might want to consider the more general situation consisting of an exact sequence of Banach algebras
\begin{equation}
 0 \longrightarrow \mathcal{S} \stackrel{\iota}{\longrightarrow}   \mathcal{B}  \stackrel{\sigma}{\longrightarrow} \mathcal{C} \longrightarrow  0;
\end{equation}
 where $\mathcal{B}$  and $\mathcal{C}$ are unitary.
\par
Let $[\mathbf{a}] \in K^{1}(\mathcal{C})$ be the $K$-theory class of  $\mathbf{a} \in GL_{n}(\mathcal{C})$. Let $A, B \in \mathcal{M}_{n}(\mathcal{B})$ such that $\sigma (A) = \mathbf{a}$  and $ \sigma (B) = \mathbf{a}^{-1}$. In analogy with the construction described in Sect. 3, we define $S_{0}= 1 - BA$  and $S_{1}= 1 - AB$. Clearly, $S_{0}, S_{1} \in \mathcal{M}_{n}(\mathcal{S})$.
Next, with these elements one manufacture the elements $\mathbf{L}, \mathbf{P}, \mathbf{R}$ and $\mathbf{e}$.
\par
Then, 
$\mathbf{R} \in \mathcal{M}_{2n}(\mathcal{S})$, \hspace{0.1cm} $\mathbf{R} = \mathbf{P} - \mathbf{e}$  
where  $\mathbf{P}$ and $\mathbf{e}$ are idempotents.
\par
We assume also that for the elements of the algebras  $\mathcal{S}, \mathcal{B}, \mathcal{C}$ there is some sort of notion of \emph{locality}, as for pseudo-differential operators, and hence that one may define \emph{local} cyclic homology of the algebra $\mathit{S}$. The notion of \emph{locality} extends to matrices with entries in these algebras. Therefore, we may define the \emph{local} Chern character $Ch (\mathbf{R}) \in H^{ \lambda, loc}_{\Lambda, ev}$ of the class $\mathbf{a}$ as explained in Sect. 8.1.
%%%%%%%%%%%%%%%%%%%%%%%%%%%%%%%%%%%%%%%%%%%%%%%%
\subsection{\emph{Local} Cyclic Homology of the Algebras of Schatten Operators}   %%%%   Sect. 10.2
\par
Proposition 24, Sect. 9 shows that the \emph{local} cyclic homology of the ideal of smoothing operators or of trace class operators is interesting. It is important to compute this \emph{local} cyclic homology.
%%%%%%%%%%%%%%%%%%%%%%%%%%%%%%%%%%%%%%%%%%%%%%
\subsection{Homological Local Index Theorem}   %%%%%%   Sect. 10.3
It is clear that our construction of the Chern character discussed in sect. 8.1 applies also in the case of $K$-homology. 
\par
As our Chern character $Ch (\mathbf{R})$ belongs to  $H^{\lambda, loc}_{even}$, and given that for its calculation it is not necessary to work with \emph{trace class} operators, we expect our considerations to apply in more general situations.
\par
Baum and Douglas \cite{Baum-Douglas} define the Chern character in $K$-\emph{homology}. It is important to study the connection between their Chern character and the Chern character we may produce as explained above.  
\par
It is an interesting problem to express  $Ch(\mathbf{R})$ in local data, where $\mathbf{R}$ is the residue operator defined in Sect. 4, i.e to prove a dual version of the Connes-Moscovici local index theorem  (Theorem 1).
%%%%%%%%%%%%%%%%%%%%%%%%%%%%%%%%%%%%%%%%%%%%%%


\begin{thebibliography}{30}   %%%%%%%%% SECTION
\par
\bibitem{Baum-Douglas} $K$-Theory and Index Theory. \emph{Operator Algebras and Applications},  Proc. Symposia Pure Math. 38, Vol. I, pp. 117-173, 1982.
\bibitem{Connes1} Connes A.: Noncommutative differential Geometry,
Publ. Math. IHES 62 (1985), pp.257 - 360
\bibitem{Connes2} Connes A.: Noncommutative Geometry, Academic Press, (1994)
\bibitem{Connes-Moscovici1} Connes A., Moscovici H.: Cyclic Cohomology, The Novikov Conjecture and Hyperbolic Groups, Topology Vol. 29, pp. 345-388, 1990.
\bibitem{Connes-Moscovici2} Connes A., Moscovici H.: The Local Index Formula in Noncommutative Geometry.
Geom. Func. Anal. 5 (1995), 174 - 243
\bibitem{Connes-Sullivan-Teleman} Connes A., Sullivan D., Teleman N.: Quasiconformal Mappings, Operators on Hilbert space and Local Formulae for Characteristic Classes, Topology Vol. 33, No. 4, pp.663-681, 1994.
\bibitem{Cuntz} Cuntz J. Cyclic Theory, Bivariant $K$-Theory and the Bivariant Chern-Connes Character. In
Cyclic Homology in Non-Commutative Geometry by J. Cuntz, G. Skandalis and B. Tsygan. Encyclopedia of Mathematical  Science, Vol. 121, Springer, 2001
\bibitem{Kubarski-Teleman} Kubarski J., Teleman N. Linear Direct Connections, Proceedings 7th Conference on
    "Geometry and Topology of Manifolds - The Mathematical Legacy of Charles Ehresmann", Betlewo, May 2005.
\bibitem{Loday} Loday J.-L.: Cyclic Homology, Grundlehren in mathematischen Wissenschaften 301, Springer Vedrlag,
Berlin Heidelberg, 1992.
\bibitem{McLane} McLane S.: Homology, Third Ed., Grundlehren der mathematischen Wissenschaften in Einzeldarstellung Band 114,
Springer Verlag, Heidelberg, 1975.
\bibitem{Mishchenko-Teleman} Mishchenko A.S., Teleman N.: Almost flat bundles and almost
flat structures. Topological Methods in Non-linear Analysis. Vol. 26, Nr. 1, pp.75-88, 2005.
\bibitem{Puschnigg} Puschnigg M.: Diffeotopy Functors of Ind-Algebras and Local Cyclic Cohomology, Documenta Mathematica Vol. 8, 143-245, 2003.
\bibitem{Skandalis} Skandalis G.: Noncommutative Geometry, the Transverse Signature Operator and Hopf Algebras.
 In Cyclic Homology in Non-Commutative Geometry by J. Cuntz, G. Skandalis and B. Tsygan. Encyclopedia of Mathematical  Science, Vol. 121, Springer, 2001.
\bibitem{Spanier} Spanier E. H.: Algebraic Topology, McGraw - Hill Series in Higher Mathematics, New York, 1966
\bibitem{Teleman1} Teleman N.:  Microlocalization de l'Homologie de Hochschild, Compt. Rend. Acad. Scie. Paris,
Vol. 326, 1261-1264, 1998.
\bibitem{Teleman2} Teleman N.: Distance Function, Linear quasi connections and Chern Character, IHES Prepublications M/04/27, June 2004.
\bibitem{Teleman3}Teleman N.: Direct Connections and Chern Character.
Proceedings of the International Conference in Honour of Jean-Paul Brasselet,
Luminy, May 2005. Will appear in World Scientific.
\bibitem{Teleman4}Teleman N.: Modified Hochschild and Periodic Cyclic Homology, IHES Prepublications, IHES/M/06/2006, December 2006.
\bibitem{Teleman5}Teleman N.: Modified Hochschild and Periodic Cyclic Homology, Central European Journal of Mathematics, will appear
\bibitem{Tsygan} Tsygan B. Cyclic Homology. In Cyclic Homology in Non-Commutative Geometry by J. Cuntz, G. Skandalis and B. Tsygan. Encyclopedia of Mathematical  Science, Vol. 121, Springer, 2001
\bibitem{Wood} Wood R.: Banach algebras and Bott periodicity. Topology 4 (1965-1966), 371-389.

\end{thebibliography}
\end{document}